\documentclass[12pt]{amsart}
\usepackage{amsmath,amsthm,amsfonts,amssymb,times,mathtools}
\usepackage{latexsym,epsfig,subfigure,afterpage} 

\usepackage{xcolor}
\usepackage[colorlinks=true, allcolors=darkblue, urlcolor=darkblue, linktocpage=true, linkcolor=darkblue, citecolor=purple]{hyperref}

\usepackage{bm}

\DeclareMathAlphabet{\curly}{U}{rsfs}{m}{n}

\theoremstyle{remark}
\newtheorem{remark}{Remark}
\newtheorem{problem}{Problem}
\theoremstyle{plain}
\newtheorem{conj}{Conjecture}
\newtheorem{defn}{Definition}

\newtheorem{lemma}{Lemma}[section]
\newtheorem{thm}{Theorem}
\newtheorem{cor}{Corollary}
\numberwithin{equation}{section}

\newcommand{\bal}{\[\begin{aligned}}
\newcommand{\eal}{\end{aligned}\]}

\newcommand{\be}{\begin{equation}}
\newcommand{\ee}{\end{equation}}

\newcommand{\benn}{\begin{equation*}}
\newcommand{\eenn}{\end{equation*}}

\renewcommand{\leq}{\leqslant}

\renewcommand{\geq}{\geqslant}

\newcommand\ba{\begin{eqnarray}}
\newcommand\ea{\end{eqnarray}}

\newcommand{\lb}{\left(}
\newcommand{\rb}{\right)}
\newcommand{\lcb}{\left\{}
\newcommand{\rcb}{\right\}}
\newcommand{\lsb}{\left[}
\newcommand{\rsb}{\right]}

\newcommand{\mrm}{\mathrm}
\newcommand{\nn}{\nonumber}

\colorlet{darkblue}{blue!70!black}
\colorlet{darkgreen}{green!70!black}
\colorlet{darkorange}{orange!70!black}
\colorlet{darkcyan}{cyan!70!black}

\def\bs#1{{\color{red}#1}}

\usepackage{graphicx}

\allowdisplaybreaks[4]         

\begin{document}

\title{The Riemann zeta function and exact exponential sum identities of divisor functions}

\author[M.~Nastasescu]{Maria Nastasescu}
\address{Department of Mathematics, 2033 Sheridan Road, Northwestern University, Evanston, IL 60207, USA}
\email{mnastase@math.northwestern.edu}

\author[N.~Robles]{Nicolas Robles}
\address{RAND Corporation, Engineering and Applied Sciences, 1200 S Hayes St, Arlington, VA 22202, USA}
\email{nrobles@rand.org}

\author[B.~Stoica]{Bogdan Stoica}
\address{Martin A. Fisher School of Physics, Brandeis University, Waltham, MA 02453, USA}
\email{bstoica@brandeis.edu}

\author[A.~Zaharescu]{Alexandru Zaharescu}
\address{Department of Mathematics, University of Illinois, 1409 West Green
Street, Urbana, IL 61801, USA; \textnormal{and} Institute of Mathematics of the Romanian academy, P.O. BOX 1-764, Bucharest, Ro-70700, Romania}
\email{zaharesc@illinois.edu}

\begin{abstract}
We prove an explicit integral formula for computing the product of two shifted Riemann zeta functions everywhere in the complex plane. We show that this formula implies the existence of infinite families of exact exponential sum identities involving the divisor functions, and we provide examples of these identities. We conjecturally propose a method to compute divisor functions by matrix inversion, without employing arithmetic techniques.

\end{abstract}


\thanks{BRX-TH-6716716}

\maketitle

\section{Introduction}

Set $s = \sigma + it$ with $\sigma, t \in \mathbb{R}$ and let
\be
\zeta(s) := \sum_{n =1}^\infty \frac{1}{n^s} = \prod_{p} \frac{1}{ 1-p^{-s}} 
\ee
for $\Re(s)>1$, and otherwise by analytic continuation, denote the Riemann zeta function. Matiyasevich introduced in a recent paper \cite{Matiyasevich} a numerical method to approximate $\zeta(s)$ to very high precision, by considering a judiciously truncated Euler product. This precision was, in turn, then used to elucidate certain numerical properties of the non-trivial zeros $\rho$ of the Riemann zeta function. A salient feature of this approach is the ability to compute with good precision the numerical values of non-trivial Riemann zeta zeros from only a few factors in the truncated Euler product. The observations originally made in \cite{Matiyasevich} were stated as conjectures, as well as questions, and some of these conjectures were proved in the affirmative by two of the current authors in \cite{NastasescuZaharescu}.

In the present paper we consider families of approximations for a product of two copies of the Riemann zeta function, one shifted by an integer $a$. These types of products have interesting arithmetic properties, since the generating function of the divisor function
\begin{align} \label{sigmaandefn}
\sigma_a(n) := \sum_{d|n} d^a,
\end{align}
with $n \in \mathbb{N}$ and $a \in \mathbb{C}$, is given by
\begin{align} \label{sigmaangenerating}
\zeta(s)\zeta(s-a) = \sum_{n=1}^\infty \frac{\sigma_a(n)}{n^s} \quad \textnormal{for} \quad \Re(s) >1,\ \Re(s) > \Re(a)+1.
\end{align}

Let $\xi$ denote the Riemann Xi function
\be
\xi(s) \coloneqq \frac{1}{2}s(s-1)\pi^{-\frac{s}{2}} \Gamma\lb\frac{s}{2}\rb \zeta(s),
\ee with $\Gamma(s)$ the Euler Gamma function. By virtue of the functional equation one has $\xi(s) = \xi(1-~s)$. In the cases $a=1,3,5$ we prove that our families indeed approximate $\zeta(s)\zeta(s-a)$ with high precision, and furthermore they induce a representation for $\zeta(s_0)\zeta(s_0-a)$ at any $s_0\in\mathbb{C}$, in terms of a vertical line integral as
\be
\xi(s_0) \xi(s_0 - a) = \frac{1}{2\pi i} \sum_{n=1}^\infty \int_{\Re(s)=\sigma}  \sigma_a(n)  \lsb \frac{g(s)g(s-a)}{s-s_0} \frac{1}{n^s}  + \frac{g(s)g(s+a)}{s-1+s_0} \frac{1}{n^{s+a}} \rsb ds,
\ee 
where the function $g(s)$ is given by
\be
\label{eq12gs}
g(s) \coloneqq \pi^{-\frac{s}{2}} (s-1)\Gamma\lb \frac{s}{2} +1 \rb,
\ee
and $\sigma$ is any real number such that $\sigma> \max\lb \Re(s_0),1-\Re(s_0) \rb $ for $a=1,3$, and in the case $a=5$ $\sigma> \max\lb \Re(s_0),1-\Re(s_0),3 \rb $.

Comparing this presentation with the usual Dirichlet series at $\Re(s)>1$, we prove infinite families of constraint relations involving the divisor functions $\sigma_a(n)$ for $a=1,3,5$. As an example, we show that the following three relations hold:
\ba
\sum_{n=0}^\infty \sigma_1(n) \lsb 2 \pi  n (2\pi n-3)+1\rsb e^{-2\pi n} &=& \frac{1}{24}, \\
\label{eq16}
\sum_{n=1}^\infty \sigma_3(n) \lb \pi n - 1 \rb e^{-2\pi n} &=& \frac{1}{240}, \\
\label{eq1p7res}
\sum_{n=1}^\infty \sigma_5(n) e^{-2\pi n } &=& \frac{1}{504}.
\ea

One immediate consequence of our relations is that we obtain results on the transcendentality of certain sums involving divisor functions. For example, in Eq. \eqref{eq16} at least one of the sums $ \sum_{n=1}^\infty \sigma_3(n) n e^{-2\pi n}$ and $ \sum_{n=1}^\infty \sigma_3(n) e^{-2\pi n}$ must be a transcendental number. Note that the analogue of this second sum for $\sigma_5(n)$ is rational, by Eq. \eqref{eq1p7res}. We will give more examples of such identities in the main body of the paper.

Finally, in Section \ref{secfinremarks} we put forward some problems and conjectures. The problems can be approached by interested readers using the techniques presented in this paper, however the conjectures may require novel ideas. In particular, we conjecture that the values of the divisor functions $\sigma_{a=1,3,5}(n)$ can be obtained by inverting certain matrices, without employing divisibility techniques. If this is true then each such inversion can be used as a primality test. We present numerical evidence for the divisor function values obtained in this manner in Tables \ref{tab2} -- \ref{tab4}.

\section{The construction}

\noindent In this section we introduce our construction. Let
\be
\zeta_p(s) \coloneqq \frac{1}{1-p^{-s}}
\ee
be the local factor of $\zeta(s)$ at place $p$. The local factor at the Archimedean place is $g(s)$ in \eqref{eq12gs}.
We consider a finite product of the local factors associated to $\zeta(s)\zeta(s-a)$. This truncated Euler product has an infinite number of poles at points $s\in\mathbb{C}$, where the reciprocals of the local factors cancel. Subtracting from the finite Euler product the principal part, defined as the sum of the principal parts at all the poles, produces the regular part. Symmetrizing this regular part under $s\to 1-s$ provides the approximation to $\zeta(s)\zeta(s-a)$. The precision with which this function estimates values of $\zeta(s) \zeta(s-a)$ depends on the number of primes used in the truncated Euler product.

More precisely, for an odd $a$ we define the finite Euler product
\be
\label{eqXiEuler}
\Xi^\mrm{Euler}_{N,a}(s)\coloneqq g(s)g(s-a) \prod_{p=2}^{p_{N}} \frac{1}{1-p^{-s}} \frac{1}{1-p^{-(s-a)}},
\ee
where $p_N$ is the cutoff in the number of  primes, and the product is understood to run only over prime numbers.

The finite Euler product $\Xi^\mrm{Euler}_{N,a}(s)$ in Eq.~\eqref{eqXiEuler} is a meromorphic function, with poles in the complex plane at the locations where the inverse local factors have zeros. Because $a$ is odd, these poles are all simple, except for the poles at $s=0$ and $s=a$. The points $s=0,a$ are poles for all finite place local factors. We denote the set of poles by $\mathfrak{O}_{N,a}$,
\begin{align}
\label{eq24poles}
\mathfrak{O}_{N,a} \coloneqq &\bigcup_{j=1}^{N} \lcb \frac{2\pi i m}{\ln p_j} |m\in\mathbb{Z} \rcb \cup \lcb a+ \frac{2\pi i m}{\ln p_j} |m\in\mathbb{Z} \rcb \nonumber \\
&\cup \lcb -2l | l\in \mathbb{N}_{>0} \rcb \cup \lb \lcb a - 2l | l\in \mathbb{N}_{> 0} \rcb - \{1 \} \rb.
\end{align}
In Eq. \eqref{eq24poles} the first two sets are the poles of the finite place local factors, and the last two sets are the poles of $g(s)g(s-a)$. Because $g(1)=0$, we see that $g(s)g(s-a)$ cannot have a pole at $s=1$.  

For $s_\star\in \mathfrak{O}_{N,a}$, we define the principal part $\Xi^{\mrm{pp},s_\star}_{N,a}(s)$ of $\Xi^\mrm{Euler}_{N,a}(s)$ at $s_\star$ as the sum of the negative powers in the Laurent expansion of $\Xi^\mrm{Euler}_{N,a}(s)$ around $s_\star$, that is
\be
\Xi^{\mrm{pp},s_\star}_{N,a}(s) \coloneqq \sum_{i=1}^{o_\star} \frac{r^{(i)}_\star}{\lb s-s_\star \rb^i },
\ee
where $o_\star$ is the order of the pole and $r^{(i)}_\star$ is the coefficient in the Laurent expansion. Then the principal part $\Xi^{\mrm{pp}}_{N,a}(s)$ of $\Xi^\mrm{Euler}_{N,a}(s)$ is defined as the sum
\be
\label{eq28Xidef}
\Xi^{\mrm{pp}}_{N,a}(s) \coloneqq \sum_{s_\star\in \mathfrak{O}_{N,a}} \Xi^{\mrm{pp},s_\star}_{N,a}(s).
\ee

\begin{remark}
The sum in $\Xi_{N,a}^\mrm{pp}(s)$ is well-defined (i.e. it converges rapidly), due to the decay of the gamma factors in the increasing $|\Im(s)|$ direction, and in the decreasing $\Re(s)$ direction.
\end{remark}

\begin{defn}
The regular part $\Xi_{N,a}^\mrm{ingoing}(s)$ of $\Xi_{N,a}^\mrm{Euler}(s)$ is
\be
\Xi^\mrm{ingoing}_{N,a}(s) \coloneqq \Xi^\mrm{Euler}_{N,a}(s) - \Xi_{N,a}^\mrm{pp}(s),
\ee
and furthermore we define 
\be
\label{hereeq28}
\Xi_{N,a}(s) \coloneqq \Xi^\mrm{ingoing}_{N,a}(s) + \Xi^\mrm{ingoing}_{N,-a}(1-s).
\ee
\end{defn}
Note that by construction $\Xi^\mrm{ingoing}_{N,a}(s)$ and $\Xi_{N,a}(s)$ are entire.

As we will show, $\Xi_{N,a}(s)$ approximates the completed zeta function product $\xi(s)\xi(s-a)$ in the limit $N\to\infty$ (we recall that $\xi(s)\coloneqq g(s) \zeta(s)$, with $g(s)$ defined in Eq. \eqref{eq12gs}).

\begin{remark}
By construction, from Eq. \eqref{hereeq28} the function $\Xi_{N,a}(s)$ satisfies the functional equation
\be
\Xi_{N,a}(s) = \Xi_{N,a}(1-s).
\ee
\end{remark}

\section{Proof of the integral formula}

To the set up the framework for our proofs we first make the following definition.

\begin{defn}
A closed contour $\mathcal{C}$ in the complex plane is sparse with respect to $\mathfrak{O}_{N,a}$ if for $s\in \mathcal{C}$ we have
\be
\min_{p\leq p_N} \min_{\substack{s_\star\in\mathfrak{O}_{N,a}\\\Im(s_\star)\neq 0}} |s-s_\star| \gg \frac{1}{p_N},
\ee
as well as
\be
\min_{\substack{s_\star \in \mathfrak{O}_{N,a}\\ \Im(s_\star) = 0}} |s-s_\star| \gg 1.
\ee 
\end{defn}

\begin{lemma}
\label{lemmaexercise}
We have
\be
\Xi_{N,a}^\mrm{pp}(s) \ll_{N,a} \frac{1}{1+|s|}.
\ee
\end{lemma}
\begin{proof}
The reader is referred to the proof of Theorem 1, relations (4.2) -- (4.7), in~\cite{NastasescuZaharescu}.
\end{proof}

\begin{thm} 
Pick integer $N>0$, and let $a$ be an odd positive integer. For any $s_0\in\mathbb{C}$, we have
\begin{align}
\label{eq32new}
\xi(s_0) \xi(s_0-a) - \Xi_{N,a}(s_0) = \frac{1}{2\pi i} \int_{\Re(s)=\sigma} &\bigg[ \frac{ \xi(s) \xi(s-a) -\Xi^\mrm{Euler}_{N,a}(s) }{s-s_0} \nonumber \\
&+ \frac{ \xi(s) \xi(s+a) -\Xi^\mrm{Euler}_{N,-a}(s) }{s-1+s_0} \bigg] ds,
\end{align}
where $\Xi_{N,a}$, $\Xi_{N,a}^\mrm{Euler}$ are given by Eqs. \eqref{hereeq28} and \eqref{eqXiEuler} respectively, and the line $\Re(s)=\sigma$ is to the right of points $s_0$, $1-s_0$.
\end{thm}

Eq. \eqref{eq32new} is an exact relation which constrains the size of the difference between the function and the approximation.

\begin{proof}
Let $\mathfrak{C}$ be a curve sparse with respect to $\mathfrak{O}_{N,a}$,  enclosing points $s_0$ and $1-s_0$, and consider a large rectangular contour $\mathcal{C}$, sparse with respect to $\mathfrak{O}_{N,a}$, enclosing curve  $\mathfrak{C}$. Applying Cauchy's residue theorem we have that
\be
\frac{1}{2\pi i}\oint_\mathcal{C} \lb \frac{\Xi^\mrm{pp}_{N,a}(s)}{s-s_0} +\frac{\Xi^\mrm{pp}_{N,-a}(s)}{s-1+s_0} \rb ds - \frac{1}{2\pi i}\oint_\mathfrak{C} \lb \frac{\Xi^\mrm{pp}_{N,a}(s)}{s-s_0} +\frac{\Xi^\mrm{pp}_{N,-a}(s)}{s-1+s_0} \rb ds = \sum_{\mathfrak{C} < \rho < \mathcal{C} } \mrm{Res} (\rho),
\ee
where the sum over residues is only over the poles between the two contours. From Lemma \ref{lemmaexercise} the integral on $\mathcal{C}$ decays to zero as the vertices of this contour are pushed to infinity, so we obtain 
\be
\label{eqisthis24}
\frac{1}{2\pi i}\oint_\mathfrak{C} \lb \frac{\Xi^\mrm{pp}_{N,a}(s)}{s-s_0} +\frac{\Xi^\mrm{pp}_{N,-a}(s)}{s-1+s_0} \rb ds  = - \sum_{\rho > \mathfrak{C}} \mrm{Res} (\rho),
\ee
where the sum over residues now runs over all the poles outside $\mathfrak{C}$.

Applying Cauchy's theorem, since $\Xi_{N,a}^\mrm{ingoing}(s)$ has no poles, each term in Eq. \eqref{hereeq28} can be written as a contour integral on $\mathfrak{C}$, and we have
\ba
\Xi_{N,a}(s_0) &=& \frac{1}{2\pi i} \oint_{\mathfrak{C}} \bigg( \frac{\Xi^\mrm{ingoing}_{N,a}(s)}{s-s_0} + \frac{\Xi^\mrm{ingoing}_{N,-a}(s)}{s-1+s_0} \bigg) ds\\
\label{eq26}
&=& \frac{1}{2\pi i} \oint_{\mathfrak{C}} \bigg( \frac{\Xi^\mrm{Euler}_{N,a}(s)}{s-s_0} + \frac{\Xi^\mrm{Euler}_{N,-a}(s)}{s-1+s_0} \bigg) ds \\
&-& \frac{1}{2\pi i} \oint_{\mathfrak{C}} \lb \frac{\Xi^\mrm{pp}_{N,a}(s)}{s-s_0} + \frac{\Xi^\mrm{pp}_{N,-a}(s)}{s-1+s_0} \rb ds. \nn
\ea
We can now apply Eq. \eqref{eqisthis24} in Eq. \eqref{eq26} to obtain
\ba
\Xi_{N,a}(s_0) &=& \frac{1}{2\pi i} \oint_{\mathfrak{C}} \bigg( \frac{\Xi^\mrm{Euler}_{N,a}(s)}{s-s_0} + \frac{\Xi^\mrm{Euler}_{N,-a}(s)}{s-1+s_0} \bigg) ds + \sum_{\rho > \mathfrak{C}} \mrm{Res} (\rho). \nn
\ea
The sum over residues outside $\mathfrak{C}$ vanishes as $\mathfrak{C}$ is pushed to infinity, due to the rapid decay of the Gamma function in the imaginary direction, and of its derivative at negative even integers in the negative real direction. In this limit Eq. \eqref{eq26} becomes 
\be
\label{eq27}
\Xi_{N,a}(s_0) = \frac{1}{2\pi i} \oint_{\mathfrak{C}} \bigg( \frac{\Xi^\mrm{Euler}_{N,a}(s)}{s-s_0} + \frac{\Xi^\mrm{Euler}_{N,-a}(s)}{s-1+s_0} \bigg) ds.
\ee
Consider now a family $\{\mathfrak{C}_l\}_{l\geq 1}$ of rectangles, with all the right sides at the same real part $\Re(s)=\sigma$. Furthermore, take the rectangles so that the right vertices get pushed to positive and negative infinity in the imaginary direction as $l$ is increased, and such that the left vertices get pushed to infinity in the imaginary and negative real directions. We keep the real part of the left side at an odd negative integer $m_l$ for all $l$, such that $|m_l|$ increases as $l$ is increased. Let $\mathfrak{C}=\mathfrak{C}_l$. There are four contributions to the the contour integral in Eq. \eqref{eq27}, coming from the four sides of the rectangle. The integrals on the horizontal sides vanish in the limit $l\to\infty$, due to the rapid decay of the Gamma function at large absolute value of the imaginary part. The contribution to the integral in Eq. \eqref{eq27} from the left vertical side at $m_l$ also vanishes in the limit $l\to\infty$, because $m_l$ is a odd negative integer of increasing absolute value. Thus, only the right vertical line contributes to Eq. \eqref{eq27} in the limit $l\to\infty$, and we arrive at 
\be
\label{eq28}
\Xi_{N,a}(s_0) = \frac{1}{2\pi i} \int_{\Re(s)=\sigma} \bigg( \frac{\Xi^\mrm{Euler}_{N,a}(s)}{s-s_0} + \frac{\Xi^\mrm{Euler}_{N,-a}(s)}{s-1+s_0} \bigg) ds.
\ee
Consider now another rectangular contour $\mathfrak{C}'$ enclosing points $s_0$ and $1-s_0$, such that the left and right vertical sides are at real parts $1-\sigma$ and $\sigma$ respectively, i.e. the contour is symmetric around the critical line. From the functional equation we have
\be
\xi(s)\xi(s-a) = \xi(1-s)\xi(1-s+a)
\ee 
so that by Cauchy's theorem we can write $\xi(s_0)\xi(s_0-a)$ as a contour integral,
\ba
\label{eq29}
\xi(s_0)\xi(s_0-a) &=& \frac{1}{2}\lsb \xi(s_0)\xi(s_0-a) + \xi(1-s_0)\xi(1-s_0+a) \rsb \\
&=& \frac{1}{4\pi i} \oint_{\mathfrak{C}'} \lsb \frac{\xi(s)\xi(s-a)}{s-s_0} + \frac{\xi(s)\xi(s+a)}{s-1+s_0} \rsb ds.
\ea
We split the contour $\mathfrak{C}'$ into parts $\mathfrak{C}'_\mathcal{L}$ and $\mathfrak{C}_\mathcal{R}'$ to the left and right of the critical line. For the left part, using the change of variables $s\to 1-s$ and the functional equation, we have
\ba
& &\frac{1}{4\pi i} \int_{\mathfrak{C}'_\mathcal{L}} \lsb \frac{\xi(s)\xi(s-a)}{s-s_0} + \frac{\xi(s)\xi(s+a)}{s-1+s_0} \rsb ds \\
&=& - \frac{1}{4\pi i} \int_{\mathfrak{C}'_\mathcal{R}} \lsb  \frac{\xi(1-s)\xi(1-s-a)}{1-s-s_0} + \frac{\xi(1-s)\xi(1-s+a)}{-s+s_0} \rsb ds \nn\\
\label{eq210}
&=& \frac{1}{4\pi i} \int_{\mathfrak{C}'_\mathcal{R}} \lsb \frac{\xi(s)\xi(s+a)}{s-1+s_0} + \frac{\xi(s)\xi(s-a)}{s-s_0} \rsb ds.
\ea
Then Eq. \eqref{eq29} becomes
\be
\xi(s_0)\xi(s_0-a) = \frac{1}{2\pi i} \int_{\mathfrak{C}'_\mathcal{R}} \lsb \frac{\xi(s)\xi(s-a)}{s-s_0} + \frac{\xi(s)\xi(s+a)}{s-1+s_0} \rsb ds.
\ee
Due to the rapid decay of the Gamma function at large absolute value of the imaginary part, the contribution of the horizontal sides in $\mathfrak{C}'_\mathcal{R}$ vanishes in the limit where the vertices of the contour are taken to have large imaginary parts in absolute value. Then we are left with the contribution of  the vertical side to the integral, and  we obtain
\ba
\label{eq212}
\xi(s_0)\xi(s_0-a) =\frac{1}{2\pi i} \int_{\Re(s)=\sigma}\lsb \frac{\xi(s)\xi(s-a)}{s-s_0} + \frac{\xi(s)\xi(s+a)}{s-1+s_0} \rsb ds.
\ea
Now we take the difference of Eqs. \eqref{eq212} and \eqref{eq28} to arrive at
\begin{align}
\xi(s_0)\xi(s_0-a) - \Xi_{N,a}(s_0) = \frac{1}{2\pi i} \int_{\Re(s)=\sigma} &\bigg[ \frac{\xi(s)\xi(s-a) - \Xi^\mrm{Euler}_{N,a}(s) }{s-s_0} \nonumber \\
&+ \frac{ \xi(s)\xi(s+a) - \Xi^\mrm{Euler}_{N,-a}(s)}{s-1+s_0} \bigg] ds,
\end{align}
which is what we wanted to prove.
\end{proof}

\begin{remark}
We remark in passing that all the steps above precisely apply to the zeta function as well. In that case, we obtain
\be
\xi(s_0) - \Xi_N(s_0) = \frac{1}{2\pi i} \int_{\Re(s)=\sigma}\lsb \xi(s) - \Xi^\mrm{Euler}_{N}(s) \rsb\lb  \frac{1}{s-s_0} + \frac{1}{s-1+s_0}\rb ds.
\ee
\end{remark}

\begin{remark}
For $a$ an odd integer we have the Legendre duplication formula
\be
\label{eq319dupformula}
g(s) g(s-a) = 2^{(a-1)/2-s}\pi^{(a+1)/2-s} (s-a)(s-a-1)\Gamma\lb s+1 \rb  \prod^{\frac{a-1}{2}}_{j=1} (s-2j-1)^{-1} ,
\ee
where the product is understood to be $1$ when $a=1$. The poles of $g(s) g(s-a)$ thus are at all negative integers and at the odd positive integers $l$, with $3\leq l\leq a-1$.
\end{remark}

A computer algebra software yields the following two lemmas.

\begin{lemma}
\label{lemmamellins}
For $s_0\in\mathbb{C}$ and $\sigma>\Re(s_0)$, $\sigma>3$, we have the inverse Mellin transforms
\ba
 \frac{J_1(N,s_0)}{\pi} &\coloneqq& \frac{1}{2\pi i} \int_{\Re(s)=\sigma} \frac{(s-2)(s-1)\Gamma(s+1)}{(s-s_0) (2\pi n)^s} ds\nn\\
  &=& \frac{1}{(2\pi n)^{s_0}} \Big[ 6 \Gamma (s_0+1,2 n \pi ) -6 \Gamma(s_0+2,2 n \pi )+\Gamma (s_0+3,2n \pi )\Big],  \\ 
\frac{J_3(N,s_0)}{2\pi^2} &\coloneqq& \frac{1}{2\pi i} \int_{\Re(s)=\sigma} \frac{(s-4)\Gamma(s+1)}{(s-s_0) (2\pi n)^s} ds \nn \\ 
&=& 2 \pi n \left[(s_0-4) E_{-s_0}(2\pi n )+e^{-2 \pi n }\right],\nn \\
\frac{J_5(N,s_0)}{4\pi^3} &\coloneqq& \frac{1}{2\pi i} \int_{\Re(s)=\sigma} \frac{(s-6)\Gamma(s+1)}{(s-3)(s-s_0) (2\pi n)^s} ds\nn \\ 
&=& \frac{(2\pi n)^4 (s_0-6) E_{-s_0}(2 \pi n )+12\pi n e^{-2 \pi  n} [\pi  n (2 \pi n+3)+3]+18 e^{-2 \pi  n}}{(2\pi n)^3(s_0-3)}. \nn
\ea
\end{lemma}

Here $\Gamma(s,a)$ is the incomplete Gamma function
\be
\Gamma(s,a) = \int_a^\infty t^{a-1} e^{-t} dt
\ee
and $E_n(z)$ is the exponential integral function
\be
E_n(z) = \int_1^\infty \frac{e^{-zt}}{t^n} dt.
\ee

\begin{lemma}
\label{lemmma3p3}
From Lemma \ref{lemmamellins}, for $s_0\in\mathbb{C}$, $\sigma>\Re(s_0)$, and $n \gg |s_0|$, we have the following asymptotic results
\begin{align*}
	\frac{1}{2\pi i} \int_{\Re(s)=\sigma} \frac{(s-2)(s-1)\Gamma(s+1)}{(s-s_0) (2\pi n)^s} ds &= (2\pi n)^2 e^{-2\pi n} \nonumber \\
	& \times \lsb 1 + \frac{s_0-4}{2\pi n} + \frac{(s_0-1)(s_0-2)}{(2\pi n)^2} + \mathcal{O} \lb \frac{(1+|s_0|)^3}{n^3} \rb \rsb, \nonumber \\
	\frac{1}{2\pi i} \int_{\Re(s)=\sigma} \frac{(s-4)\Gamma(s+1)}{(s-s_0) (2\pi n)^s} ds &= (2\pi n) e^{-2\pi n} \nonumber \\
	&\times \lsb 1 + \frac{s_0-4}{2\pi n} + \frac{s_0(s_0-4)}{(2\pi n)^2} + \mathcal{O} \lb \frac{(1+|s_0|)^3}{n^3} \rb \rsb, \nonumber \\
	\frac{1}{2\pi i} \int_{\Re(s)=\sigma} \frac{(s-6)\Gamma(s+1)}{(s-3)(s-s_0) (2\pi n)^s} ds &= e^{-2\pi n} \nonumber \\
	&\times \lsb 1 + \frac{s_0-3}{2\pi n} + \frac{s_0(s_0-4)-6}{(2\pi n )^2} + \mathcal{O} \lb \frac{(1+|s_0|)^3}{n^3} \rb \rsb. \nonumber
\end{align*}
\end{lemma}

For a prime $p$, let $\{ \leq p \}$ denote the union of the infinite set of positive integers that have all prime factors less than or equal to $p$, and $\{1\}$. Lemmas \ref{lemmamellins}, \ref{lemmma3p3} above can be used to prove the following result.

\begin{lemma} 
\label{lemmma3p4}
Let $\sigma_a(n)\coloneqq \sum_{d|n} d^a$ be the a-th divisor function. For integer $N\geq 1$, $s_0\in\mathbb{C}$ and $a=1,3,5$, we have
\be
\xi(s_0)\xi(s_0-a) - \Xi_{N,a}(s_0) =  \sum_{\substack{ n\geq p_{N+1} \\ n \notin \{ \leq p_N \} }} \sigma_a(n) \lsb  J_a(N,s_0) + J_a(N,a+1-s_0) \rsb,
\ee
where the functions $J_a(N,s_0)$ are given in Lemma \ref{lemmamellins}.
\end{lemma}

\begin{proof}
We have
\ba
\xi(s) \xi(s-a) &=& g(s) g(s-a) \sum_{n=1}^\infty \frac{\sigma_a(n)}{n^s}, \\
\Xi^\mrm{Euler}_{N,a}(s) &=& g(s)g(s-a) \prod_{p=2}^{p_{N}} \frac{1}{1-p^{-s}} \frac{1}{1-p^{-(s-a)}} \\ \nonumber
 &=& g(s) g(s-a) \sum_{n \in \{ \leq p_N \}} \frac{\sigma_a(n)}{n^s}.  
\ea
We consider the difference in Eq. \eqref{eq32new} for $\sigma>1$,
\ba
\nn \xi(s_0)\xi(s_0-a) - \Xi_{N,a}(s_0) &=& \frac{1}{2\pi i} \int_{\Re(s)=\sigma} \frac{g(s)g(s-a) }{s-s_0} \bigg( \sum_{n=1}^\infty - \sum_{n \in \{ \leq p_N \}} \bigg) \frac{\sigma_a(n)}{n^s} ds \\
&+& \frac{1}{2\pi i} \int_{\Re(s)=\sigma} \frac{ g(s)g(s+a)}{s-1+s_0} \bigg( \sum_{n=1}^\infty - \sum_{n \in \{ \leq p_N \}} \bigg) \frac{\sigma_{-a}(n)}{n^s} ds .
\label{eq323}
\ea
We denote the second integral above by $I_2(N)$. Performing a change of variables $s'\coloneqq s + a$ and using $\sigma_{-a}(n)n^a = \sigma_a(n)$,  we have
\ba
I_2(N) &=& \frac{1}{2\pi i} \int_{\Re(s)=\sigma} \frac{ g(s)g(s+a)}{s-1+s_0} \bigg( \sum_{n=1}^\infty - \sum_{n \in \{ \leq p_N \}} \bigg) \frac{\sigma_{-a}(n)}{n^s} ds \nonumber \\
&=&  \frac{1}{2\pi i} \int_{\Re(s')=\sigma+a} \frac{ g(s'-a)g(s')}{s'-a-1+s_0} \bigg( \sum_{n=1}^\infty - \sum_{n \in \{ \leq p_N \}} \bigg) \frac{\sigma_{a}(n)}{n^{s'}} ds',
\ea
so that Eq. \eqref{eq323} becomes
\ba
\nn \xi(s_0)\xi(s_0-a) - \Xi_{N,a}(s_0) &=& \frac{1}{2\pi i} \int_{\Re(s)=\sigma} \frac{g(s)g(s-a) }{s-s_0} \bigg( \sum_{n=1}^\infty - \sum_{n \in \{ \leq p_N \}} \bigg) \frac{\sigma_a(n)}{n^s} ds \\
&+& \frac{1}{2\pi i} \int_{\Re(s)=\sigma+a} \frac{ g(s)g(s-a)}{s-a-1+s_0} \bigg( \sum_{n=1}^\infty - \sum_{n \in \{ \leq p_N \}} \bigg) \frac{\sigma_{a}(n)}{n^s} ds .
\label{eq326}
\ea
From the duplication formula \eqref{eq319dupformula} for $a=1,3,5$ we have
\ba
g(s) g(s-1) &=& \pi \frac{ (s-2) (s-1) \Gamma (s+1)}{ \lb 2\pi\rb^{s} }, \\
g(s) g(s-3) &=& 2 \pi^2 \frac{(s-4) \Gamma (s+1)}{ \lb 2\pi\rb^{s} }, \\
g(s) g(s-5) &=& 4\pi^3 \frac{ (s-6) \Gamma(s+1) }{ \lb 2\pi\rb^{s} \lb s-3 \rb}. 
\ea
Using Lemma \ref{lemmamellins} and Eq. \eqref{eq326}, and recalling the notation
\ba
J_1(N,s_0) &=&  \frac{\pi}{(2\pi n)^{s_0}} \big[ 6 \Gamma (s_0+1,2 n \pi ) -6 \Gamma(s_0+2,2 n \pi )+\Gamma (s_0+3,2n \pi )\big], \nn \\
J_3(N,s_0) &=& 4 \pi^3 n \left[(s_0-4) E_{-s_0}(2\pi n )+e^{-2 \pi n }\right], \nn \\
J_5(N,s_0) &=& 4\pi^3 \frac{(2\pi n)^4 (s_0-6) E_{-s_0}(2 \pi n )+12\pi n e^{-2 \pi  n} [\pi  n (2 \pi n+3)+3]+18 e^{-2 \pi  n}}{(2\pi n)^3(s_0-3)}, \nn
\ea
we thus have, for $a=1,3,5$, 
\ba
& & \xi(s_0)\xi(s_0-a) - \Xi_{N,a}(s_0) =  \bigg( \sum_{n=1}^\infty - \sum_{n \in \{ \leq p_N \}} \bigg) \sigma_a(n) \lsb  J_a(N,s_0) + J_a(N,a+1-s_0) \rsb.  \nn 
\ea
We note that the contributions of all the integers $n$ less that $p_{N+1}$ cancel between the two sums, and therefore we obtain
\be
\sum_{n=1}^\infty - \sum_{n \in \{ \leq p_N \}} = \sum_{\substack{ n\geq p_{N+1} \\ n \notin \{ \leq p_N \} }},
\ee
which leads to the desired result.
\end{proof}

Lemmas \ref{lemmma3p3} and \ref{lemmma3p4} immediately lead to the following corollary. 

\begin{cor}
\label{corr1}
In the large $|s_0|/p_{N+1}$ limit, for $a=1$ we have
\begin{align}
	\xi(s_0)\xi(s_0-1) - \Xi_{N,1}(s_0) &= 2 \pi \sum_{\substack{ n\geq p_{N+1} \\ n \notin \{ \leq p_N \} }} \sigma_1(n) \lb 2\pi n \rb^2 e^{-2\pi n} \nonumber \\
	&\quad \times \bigg[ 1 - \frac{ 3 }{ 2\pi n} + \frac{(s_0-1)^2}{(2 \pi n)^2} + \mathcal{O} \bigg( \frac{\lb 1 + |s_0| \rb^3}{n^3} \bigg) \bigg],
\end{align}
for $a=3$ we have
\begin{align}
\xi(s_0)\xi(s_0-3) - \Xi_{N,3}(s_0) &= 4\pi^2 \sum_{\substack{ n\geq p_{N+1} \\ n \notin \{ \leq p_N \} }} \sigma_3(n) (2\pi n)e^{-2\pi n} \nn \\
& \quad \times \bigg[ 1 - \frac{2}{ 2\pi n} + \frac{s_0(s_0-4)}{ \lb 2 \pi  n\rb^2} + \mathcal{O}\bigg( \frac{\lb 1 + |s_0| \rb^3}{n^3} \bigg) \bigg],
\end{align}
and lastly for $a=5$ we have
\begin{align}
	\xi(s_0)\xi(s_0-5) - \Xi_{N,5}(s_0) &= 8\pi^3 \sum_{\substack{ n\geq p_{N+1} \\ n \notin \{ \leq p_N \} }} \sigma_5(n) e^{-2\pi n} \nonumber \\
	&\quad \times \bigg[ 1 + \frac{s_0(s_0-6)}{(2 \pi n)^2} + \mathcal{O} \bigg( \frac{\lb 1 + |s_0| \rb^3}{n^3} \bigg) \bigg].
\end{align}
\end{cor}

\begin{thm}
\label{thm2mainthm}
Let $\xi(s)$ be the Riemann Xi function. For any $\mathcal{R}>0$, there exist constants 
\begin{align*}
B_1(\mathcal{R}),B_2(\mathcal{R}), \dots, B_6(\mathcal{R})
\end{align*}
such that for any $N \geq 1$ we have
\begin{align}
	\left| \xi(s)\xi(s-1) - \Xi_{N,1}(s) -(2\pi p_{N+1})^3 e^{-2\pi p_{N+1}} \right| &\leq B_{1}(\mathcal{R}) p^2_{N+1} e^{-2\pi p_{N+1}} \nonumber \\
	&\quad + B_2(\mathcal{R}) p_{N+2}^3 e^{-2\pi p_{N+2}}, \nonumber \\
	\left| \xi(s)\xi(s-3) - \Xi_{N,3}(s) - (2\pi)^3 p_{N+1}^4  e^{-2\pi p_{N+1}} \right| &\leq B_3(\mathcal{R}) p^3_{N+1} e^{-2\pi p_{N+1}} \nonumber \\
	&\quad + B_4(\mathcal{R}) p_{N+2}^4 e^{-2\pi p_{N+2}}, \nonumber \\
	\left| \xi(s) \xi(s-5) - \Xi_{N,5}(s) - (2\pi)^3 p^5_{N+1} e^{-2\pi p_{N+1}} \right| &\leq B_5(\mathcal{R}) p_{N+1}^3 e^{-2\pi p_{N+1}} \nonumber \\
	&\quad + B_6(\mathcal{R}) p_{N+2}^5 e^{-2\pi p_{N+2}}, \nn
\end{align}
uniformly for all $s\in\mathcal{C}$ with $|s|\leq \mathcal{R}$, where functions $\Xi_{N,a=1,3,5}(s)$ are defined in Eq. \eqref{hereeq28}.
\end{thm}

\begin{proof}
Immediately follows from Corollary \ref{corr1}.
\end{proof}

\begin{cor}
For $a=1,3$, $\sigma > \Re(s_0)$, $\sigma>\Re(1-s_0)$, and any $s_0 \in \mathbb{C}$, the product $\xi(s_0)\xi(s_0-a)$ has an integral representation as
\be
\label{eq336}
\xi(s_0) \xi(s_0 - a) = \frac{1}{2\pi i} \sum_{n=1}^\infty \int_{\Re(s)=\sigma}  \sigma_a(n)  \lsb \frac{g(s)g(s-a)}{s-s_0} \frac{1}{n^s}  + \frac{g(s)g(s+a)}{s-1+s_0} \frac{1}{n^{s+a}} \rsb ds.
\ee
For $a=5$ Eq. \eqref{eq336} still holds, with the additional restriction that $\sigma>3$, to stay to the right of the $s=3$ pole in $g(s-5)$.
\end{cor}

\begin{proof}
For $a=1,3$, consider real $\sigma'>1$ such that $\sigma' > \Re(s_0)$, $\sigma' > \Re(1-s_0)$. From Theorem \ref{thm2mainthm}, the product $\xi(s)\xi(s-a)$ has an integral representation given by Eq. \eqref{eq28},
\be
\xi(s)\xi(s-a) = \frac{1}{2\pi i} \int_{\Re(s)=\sigma'} \bigg( \frac{\Xi^\mrm{Euler}_{N,a}(s)}{s-s_0} + \frac{\Xi^\mrm{Euler}_{N,-a}(s)}{s-1+s_0} \bigg) ds,
\ee
for any $s_0\in\mathbb{C}$. Because $\sigma' > 1$, the Euler product can be written as a Dirichlet series,
\ba
\nn \xi(s)\xi(s-a) &=& \frac{1}{2\pi i} \int_{\Re(s)=\sigma'} \sum_{n=1}^\infty \left[ \frac{g(s)g(s-k)}{s-s_0} \frac{\sigma_a(n)}{n^s}+ \frac{g(s)g(s+a)}{s-1+s_0} \frac{\sigma_{-a}(n)}{n^s} \right] ds \\
\label{eq338sumint}
&=& \frac{1}{2\pi i} \int_{\Re(s)=\sigma'} \sum_{n=1}^\infty \sigma_{a}(n) \left[ \frac{g(s)g(s-k)}{s-s_0} \frac{1}{n^s}+ \frac{g(s)g(s+a)}{s-1+s_0} \frac{1}{n^{s+a}} \right] ds.
\ea 
This proves the corollary for $a=1,3$ in the region $\sigma=\sigma'>1$. To prove it in the region $ \max \lb \Re(s_0), \Re(1-s_0) \rb < \sigma \leq~1$ (assuming that $s_0$ is in the critical strip), note that for $\sigma'>1$ we can exchange the sum and integral in Eq. \eqref{eq338sumint},
\be
\xi(s)\xi(s-a) = \frac{1}{2\pi i} \sum_{n=1}^\infty \int_{\Re(s)=\sigma'} \sigma_{a}(n) \left[ \frac{g(s)g(s-k)}{s-s_0} \frac{1}{n^s}+ \frac{g(s)g(s+a)}{s-1+s_0} \frac{1}{n^{s+a}} \right] ds.
\ee
Now,
\ba
& &\int_{\Re(s)=\sigma} \sigma_{a}(n) \left[ \frac{g(s)g(s-k)}{s-s_0} \frac{1}{n^s}+ \frac{g(s)g(s+a)}{s-1+s_0} \frac{1}{n^{s+a}} \right] ds \nn\\
&=&\int_{\Re(s)=\sigma'} \sigma_{a}(n) \left[ \frac{g(s)g(s-k)}{s-s_0} \frac{1}{n^s}+ \frac{g(s)g(s+a)}{s-1+s_0} \frac{1}{n^{s+a}} \right] ds
\ea
because there are no poles between the vertical lines at $\sigma$ and $\sigma'$, and the result follows.

For $a=5$ and real $\sigma'>3$ such that $\sigma' > \Re(s_0)$, $\sigma' > \Re(1-s_0)$, the steps leading up to and including Eq. \eqref{eq338sumint} still hold.
\end{proof}

\begin{remark}
All the steps above apply to the Riemann Xi function as well, and in that case we obtain
\be
\xi(s_0) = \frac{1}{2\pi i} \sum_{n=1}^\infty \int_{\Re(s)=\sigma}  \frac{g(s)}{n^s}  \lb \frac{1}{s-s_0} + \frac{1}{s-1+s_0} \rb ds
\ee
for all $s_0\in\mathbb{C}$ and $\sigma>\max\lb\Re(s_0),1-\Re(s_0)\rb$.
\end{remark}

\section{Polynomial constraint relations for the divisor function}

In this section we will derive constraint relations on the divisor function. The strategy is to integrate the presentation \eqref{eq336} for $\xi(s)\xi(s-a)$ on a vertical line between $\rho-iT$ and $\rho+iT$ for $\rho>3$, and to compare this integral with the integral of the usual presentation of $\xi(s)\xi(s-a)$.

\begin{thm}
\label{thm3ishere}
For all integers $k\geq 0$ and with the divisor function $\sigma_a(n)$ defined as in Eq. \eqref{sigmaandefn}, we~have
\begin{align} \label{eq41theseareconstraints}
	\sum_{n=1}^\infty \sigma_1(n) \bigg[ (2\pi n)^{k+1} \, _1F_1(-k-3;k+2;2\pi n)-\frac{(2\pi n)^{k+2} \,
   _1F_1(-k-2;k+3;2\pi n)}{k+2} \bigg] e^{-2\pi n} &= 0, \nonumber \\
   \sum_{n=1}^\infty \sigma_3(n) \lsb (2\pi n)^{k+1} \,
   _1F_1(-k-2;k+5;2\pi n)-\frac{ (2\pi n)^{k+2} \,_1F_1(-k-1;k+6;2\pi n)}{k+5} \rsb e^{-2\pi n } &=0,  \nonumber \\
   \sum_{n=1}^\infty \sigma_5(n) \lb 2\pi n \rb^{k+1} \, _1F_1(-k-1;k+7;2\pi n) \, e^{-2\pi n } &= 0,
\end{align}
where $_1 F _1(a,b,z)$ is the confluent hypergeometric function.
\end{thm}

We plot the summation kernel multiplying $\sigma_1(n)$ on the first line of Eq. \eqref{eq41theseareconstraints} in Figure \ref{fig1}.

Before proving Theorem \ref{thm3ishere}, we will give a few remarks.

\begin{remark}
The series expansions of the hypergeometric functions terminate, and the summation kernels in Eqs. \eqref{eq41theseareconstraints} are polynomials in $2\pi n$, of degrees $2k+4$, $2k+3$ and $2k+2$  respectively. In terms of these polynomials the constraints \eqref{eq41theseareconstraints} are
\ba
\label{eq42threecases}
\sum_{n=1}^\infty \sigma_1(n) \sum_{j=0}^{k+3} \frac{(-1)^{j}(k+j+3)(2\pi n)^{k+j+1}}{j!(k-j+3)! (k+j+1)!} e^{-2\pi n} &=& 0, \nn \\
\sum_{n=1}^\infty \sigma_3(n) \sum_{j=0}^{k+2} \frac{(-1)^{j} (k+j+2)(2\pi n)^{k+j+1}}{j! (k-j+2)! (k+j+4)!} e^{-2\pi n } &=&0, \\
\sum_{n=1}^\infty \sigma_5(n) \sum_{j=0}^{k+1} \frac{(-1)^{j} (2\pi n)^{k+j+1}}{j! (k-j+1)! (k+j+6)!} e^{-2\pi n } &=& 0,\nn
\ea
for integer $k\geq 0$. For example, for $k=0$ these relations reduce to
\ba
\sum_{n=1}^\infty \sigma_1(n) \left(2 \pi ^3 n^3-10 \pi ^2 n^2+12 \pi  n-3\right) n e^{-2\pi n} &=& 0,\\
\sum_{n=1}^\infty \sigma_3(n)  \left(4 \pi ^2 n^2-18 \pi  n+15\right) n e^{-2\pi n } &=&0, \\
\label{eq45transcend}
\sum_{n=1}^\infty \sigma_5(n) (2 \pi  n-7) n e^{-2\pi n } &=& 0.
\ea

These relations give further transcendentality results. For example, in Eq. \eqref{eq45transcend} at least one of the two sums $\sum_{n=1}^\infty \sigma_5(n) n^2 e^{-2\pi n}$, $\sum_{n=1}^\infty \sigma_5(n) n e^{-2\pi n}$ must be transcendental.

\end{remark}

\begin{figure}[htp]
\centering
\includegraphics[scale=0.64]{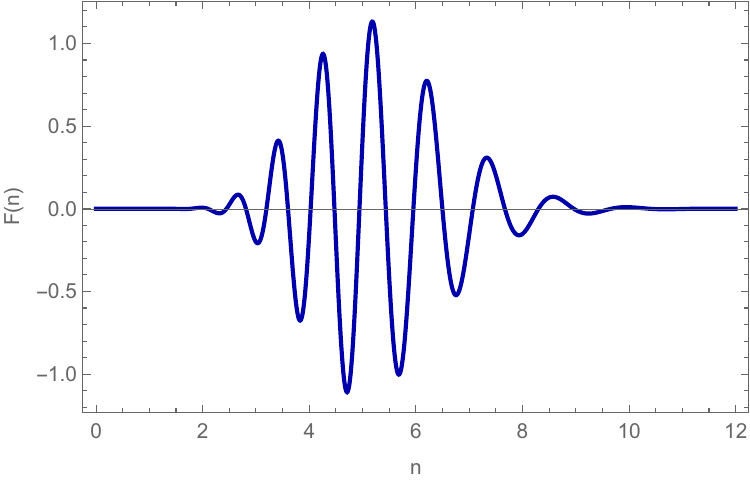} 
\includegraphics[scale=0.64]{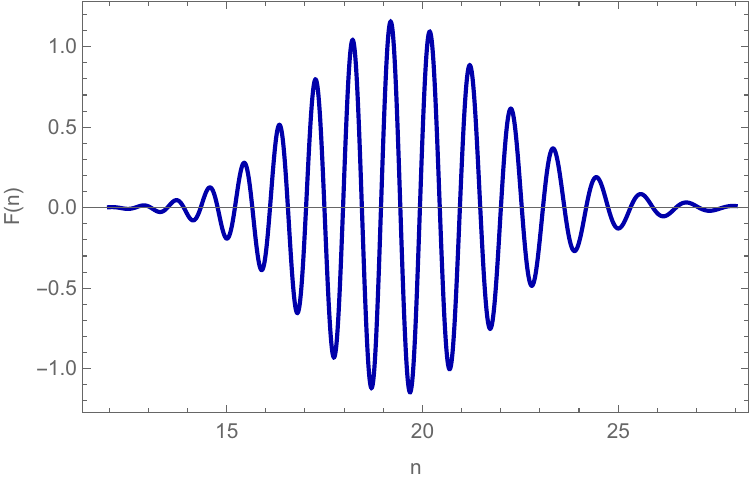} 
\caption{The summation kernel $F(n)$ multiplying $\sigma_1(n)$ in the first line of Eqs.~\eqref{eq41theseareconstraints} and \eqref{eq42threecases}, as a function of $n$, for $k=30$ and $k=120$ respectively. The overall normalizations are arbitrary. Note that $F(n)$ evaluated only at integer values of $n$ contributes to Eqs. \eqref{eq41theseareconstraints}, \eqref{eq42threecases}. }
\label{fig1}
\end{figure}

We are now in a position to prove Theorem \ref{thm3ishere}. 

\begin{proof}[Proof of Theorem \textnormal{\ref{thm3ishere}}]
For $a=1,3,5$, $\rho>1$ and $T>0$ let
\be
\label{eq41start}
I(\rho,T,a) \coloneqq \int_{t=-T}^T \xi(\rho+it) \xi(\rho+it-a) dt.
\ee
From Eq. \eqref{eq336}, for $\sigma>\rho$, $\sigma>3$ this is
\ba
I(\rho,T,a) &=& \frac{1}{2\pi i} \int_{t=-T}^T \sum_{n=1}^\infty \int_{\Re(s)=\sigma}  \sigma_a(n)  \lsb \frac{g(s)g(s-a)}{s-\rho-it} \frac{1}{n^s}  + \frac{g(s)g(s+a)}{s-1+\rho+it} \frac{1}{n^{s+a}} \rsb ds dt \nn\\
&=& \frac{1}{2\pi i} \sum_{n=1}^\infty \sigma_a(n) \int_{\Re(s)=\sigma}  \int_{t=-T}^T  \lsb \frac{g(s)g(s-a)}{s-\rho-it} \frac{1}{n^s}  + \frac{g(s)g(s+a)}{s-1+\rho+it} \frac{1}{n^{s+a}} \rsb dt ds,
\ea
where we have used Fubini's theorem to interchange order of the sum and integrals. We now introduce another cutoff, $0<T_0\ll T$, and we let
\begin{align}
I^\mrm{nonpert,1}(\rho,T_0,T,a) \coloneqq \frac{1}{2\pi i} \sum_{n=1}^\infty \sigma_a(n) \int_{\substack{\Re(s)=\sigma\\|\Im(s)|>T_0}}  \int_{t=-T}^T  \bigg[ &\frac{g(s)g(s-a)}{s-\rho-it} \frac{1}{n^s} \nonumber \\ 
+ &\frac{g(s)g(s+a)}{s-1+\rho+it} \frac{1}{n^{s+a}} \bigg] dt ds,
\end{align}
so that
\ba
I(\rho,T,a) &=& \frac{1}{2\pi i} \sum_{n=1}^\infty \sigma_a(n) \int_{\substack{\Re(s)=\sigma\\|\Im(s)|<T_0}}  \int_{t=-T}^T  \lsb \frac{g(s)g(s-a)}{s-\rho-it} \frac{1}{n^s}  + \frac{g(s)g(s+a)}{s-1+\rho+it} \frac{1}{n^{s+a}} \rsb dt ds \nn \\
&+& I^\mrm{nonpert,1}(\rho,T_0,T,a).
\label{eq4444}
\ea
We perform the $t$ integrals first. These yield
\ba
\int_{t=-T}^T \frac{dt}{s-\rho-it} &=& i\lsb \ln \lb s - \rho - iT \rb - \ln\lb s - \rho + iT \rb\rsb, \\
\int_{t=-T}^T \frac{dt}{s-1+\rho+it} &=& - i \lsb \ln \lb s-1 + \rho + i T  \rb - \ln \lb s-1 + \rho - i T  \rb  \rsb.
\ea
Expanding in the large $T$ limit (and using that $T_0 \ll T$) we thus have
\ba
\int_{t=-T}^T \frac{dt}{s-\rho-it} &=& \pi + 2 \sum_{k=0}^\infty  \frac{(-1)^{k+1}(s-\rho)^{2k+1}}{(2k+1)T^{2k+1}} , \\
\int_{t=-T}^T \frac{dt}{s-1+\rho+it} &=& \pi + 2 \sum_{k=0}^\infty  \frac{(-1)^{k+1}(s-1+\rho)^{2k+1}}{(2k+1)T^{2k+1}}.
\ea
Eq. \eqref{eq4444} then becomes
\ba
I(\rho,T,a) &=& \frac{1}{2\pi i} \sum_{n=1}^\infty \sigma_a(n) \int_{\substack{\Re(s)=\sigma\\|\Im(s)|<T_0}}  \Bigg[ \frac{g(s)g(s-a)}{n^s} \lb  \pi + 2 \sum_{k=0}^\infty  \frac{(-1)^{k+1}(s-\rho)^{2k+1}}{(2k+1)T^{2k+1}} \rb \nn \\
&+& \frac{g(s)g(s+a)}{n^{s+a}} \lb \pi + 2 \sum_{k=0}^\infty  \frac{(-1)^{k+1}(s-1+\rho)^{2k+1}}{(2k+1)T^{2k+1}} \rb \Bigg] ds +I^\mrm{nonpert,1}(\rho,T_0,T,a).
\ea
With the notation
\begin{align}
I^\mrm{nonpert,2}(\rho,T_0,T,a) &\coloneqq \frac{1}{2\pi i} \sum_{n=1}^\infty \sigma_a(n) \int_{\substack{\Re(s)=\sigma\\|\Im(s)|>T_0}}  \bigg[ \frac{g(s)g(s-a)}{n^s} \bigg(  \pi + 2 \sum_{k=0}^\infty  \frac{(-1)^{k+1}(s-\rho)^{2k+1}}{(2k+1)T^{2k+1}} \bigg) \nn \\
&\quad + \frac{g(s)g(s+a)}{n^{s+a}} \bigg( \pi + 2 \sum_{k=0}^\infty  \frac{(-1)^{k+1}(s-1+\rho)^{2k+1}}{(2k+1)T^{2k+1}} \bigg) \bigg] ds
\end{align}
we arrive at
\ba
I(\rho,T,a) &=& \frac{1}{2\pi i} \sum_{n=1}^\infty \sigma_a(n) \int_{\Re(s)=\sigma}  \Bigg[ \frac{g(s)g(s-a)}{n^s} \lb  \pi + 2 \sum_{k=0}^\infty  \frac{(-1)^{k+1}(s-\rho)^{2k+1}}{(2k+1)T^{2k+1}} \rb \nn \\
&+& \frac{g(s)g(s+a)}{n^{s+a}} \lb \pi + 2 \sum_{k=0}^\infty  \frac{(-1)^{k+1}(s-1+\rho)^{2k+1}}{(2k+1)T^{2k+1}} \rb \Bigg] ds \nn\\
&+& I^\mrm{nonpert,1}(\rho,T_0,T,a) - I^\mrm{nonpert,2}(\rho,T_0,T,a).
\label{eq411comp}
\ea
Because $I(\rho,T,a)$ is the integral of $\xi(s)\xi(s-a)$ on a vertical line in a right half-plane (see Eq.~\eqref{eq41start}), for $\rho>a+1$ it also has a presentation as
\ba
I(\rho,T,a) &=& \int_{t=-T}^T \lb \sum_{n=1}^\infty \sigma_a(n) \frac{g(\rho+it)g(\rho+it-a)}{n^{\rho+it}} \rb dt \nn \\
&=& \int_{t=-\infty}^\infty  \lb \sum_{n=1}^\infty \sigma_a(n) \frac{g(\rho+it)g(\rho+it-a)}{n^{\rho+it}} \rb dt - I^\mrm{nonpert,3}(\rho,T,a),
\label{eq412comp}
\ea
where 
\be
I^\mrm{nonpert,3}(\rho,T,a) \coloneqq \lb \int_{t=-\infty}^{-T} + \int_{t=T}^{\infty} \rb \bigg( \sum_{n=1}^\infty \sigma_a(n) \frac{g(\rho+it)g(\rho+it-a)}{n^{\rho+it}} \bigg) dt.
\ee
Let now $T_0 = \sqrt{T}$, and note that the terms $I^\mrm{nonpert,1,2}(\rho,\sqrt{T},T,a)$, $I^\mrm{nonpert,3}(\rho,T,a)$ do not contribute at any finite order in the Taylor series in $1/T$, in the sense that they do not have a $1/T^m$ contribution for any integer $m>0$. Furthermore, from Fubini's theorem we can interchange the integral and summations in Eq. \eqref{eq411comp}.  Reconciling the Taylor series expansion in $1/T$ of Eqs.~\eqref{eq411comp}, \eqref{eq412comp}, we thus arrive at
\ba
\label{constraint0}
\frac{1}{2} \sum_{n=1}^\infty \sigma_a(n) \int_{\Re(s)=\sigma}  \Bigg[ \frac{g(s)g(s-a)}{n^s} 
+ \frac{g(s)g(s+a)}{n^{s+a}}\Bigg] ds = \int_{\Re(s)=\rho}  \lb \sum_{n=1}^\infty \sigma_a(n) \frac{g(s)g(s-a)}{n^{s}} \rb ds \nn\\
\ea
and, for all integers $k\geq 0$,
\ba
Q(\rho,a,k) &\coloneqq& \sum_{n=1}^\infty \sigma_a(n) \int_{\Re(s)=\sigma}  \Bigg[ \frac{g(s)g(s-a)}{n^s} (s-\rho)^{2k+1} + \frac{g(s)g(s+a)}{n^{s+a}} (s-1+\rho)^{2k+1} \Bigg] ds \nn \\
&=&0.
\label{condition415}
\ea

Eq. \eqref{constraint0} is automatically satisfied, because one can perform the variable change $s\to s-a$ and then move the $\Re(s)=\rho$ line to $\Re(s) = \sigma+a$.  The constraints follow from Eq. \eqref{condition415}, $Q(\rho,a,k)=0$, which, using that $2k+1$ is odd, is
\be
\label{eq416constrs}
\sum_{n=1}^\infty \sigma_a(n) \int_{\Re(s)=\sigma} \sum_{l=0}^{2k+1} \binom{2k+1}{l} \Bigg[ - \frac{g(s)g(s-a)}{n^s} (-s)^{l} + \frac{g(s)g(s+a)}{n^{s+a}} (s-1)^{l} \Bigg] \rho^{2k+1-l} ds =0.
\ee
In Eq. \eqref{eq416constrs} $\rho$ is a free parameter, so the coefficient multiplying each power of $\rho$ must vanish. Using Fubini's theorem, we must thus have
\be
\label{eq417}
\sum_{n=1}^\infty \sigma_a(n) \int_{\Re(s)=\sigma} \Bigg[ - \frac{g(s)g(s-a)}{n^s} (-s)^{k} + \frac{g(s)g(s+a)}{n^{s+a}} (s-1)^{k} \Bigg] ds =0,
\ee
for all integers $k\geq 0$.

Performing the contour integral, for $a=1$ the set of constraints \eqref{eq417} is equivalent to 
\be
\sum_{n=1}^\infty \sigma_1(n) \lsb (2\pi n)^{k+1} \, _1F_1(-k-3;k+2;2\pi n)-\frac{(2\pi n)^{k+2} \,
   _1F_1(-k-2;k+3;2\pi n)}{k+2} \rsb e^{-2\pi n} = 0,
\ee
for all integers $k\geq0$. For $a=3$ the constraints \eqref{eq417} are equivalent to
\be
\sum_{n=1}^\infty \sigma_3(n) \lsb (2\pi n)^{k+1} \,
   _1F_1(-k-2;k+5;2\pi n)-\frac{ (2\pi n)^{k+2} \,_1F_1(-k-1;k+6;2\pi n)}{k+5} \rsb e^{-2\pi n } =0,
\ee
and for $a=5$ they are equivalent to
\be
\sum_{n=1}^\infty \sigma_5(n) \lb 2\pi n \rb^{k+1} \, _1F_1(-k-1;k+7;2\pi n) \, e^{-2\pi n } = 0,
\ee
for all integers $k\geq 0$. This completes the proof.
\end{proof}

As a byproduct of our work, we obtain the results in Corollary \ref{cor3}. Note that in each case, the summand is positive for all integers $n\geq 1$.

\begin{cor}
\label{cor3}
With $\sigma_a(n)$ the divisor function as defined in Eq. \eqref{sigmaandefn}, we have
\ba
\label{eq419const1}
\sum_{n=0}^\infty \sigma_1(n) \lsb 2 \pi  n (2\pi n-3)+1\rsb e^{-2\pi n} &=& \frac{1}{24}, \\
\sum_{n=1}^\infty \sigma_3(n) \lb \pi n - 1 \rb e^{-2\pi n} &=& \frac{1}{240}, \\
\label{eq419const3}
\sum_{n=1}^\infty \sigma_5(n) e^{-2\pi n } &=& \frac{1}{504}.
\ea
\end{cor}

Eqs. \eqref{eq419const1} -- \eqref{eq419const3} can be derived by applying our main result in Eq. \eqref{eq336} at $s_0=0$ and then computing the inverse Mellin transform. We note that the left-hand sides of expressions \eqref{eq419const1} -- \eqref{eq419const3} formally match the left-hand sides of the constraints \eqref{eq41theseareconstraints} for $k=-1$.

As we will see in the next section, constraints of the type \eqref{eq419const1} -- \eqref{eq419const3} can be put forward also for odd $a>5$. However, the summands multiplying the divisor function become more complicated as $a$ increases past $a=5$.

\section{Final remarks}
\label{secfinremarks}

For the purposes of illustration we have focused on the cases $a \in \{1,3,5\}$ and produced the exponential sums twisted by generalized divisor functions appearing in Theorem~\ref{thm3ishere} and  Corollary~\ref{cor3}. However, it should be clear that many of the results in this paper can be conjectured to generalize beyond the cases we have proven.

\begin{conj}
For all $s_0\in\mathbb{C}$ and all $a \in \mathbb{C}$ we have
\be
\label{eq51}
\xi(s_0) \xi(s_0 - a) = \frac{1}{2\pi i} \sum_{n=1}^\infty \int_{\Re(s)=\sigma}  \sigma_a(n)  \lsb \frac{g(s)g(s-a)}{s-s_0} \frac{1}{n^s}  + \frac{g(s)g(s+a)}{s-1+s_0} \frac{1}{n^{s+a}} \rsb ds
\ee
for  $\sigma>\max\lb\left|\Re(a)\right|-2,\, \Re(s_0),\,\Re(1-s_0) \rb$ (so that all the poles are to the left of the integration line).
\end{conj}

\begin{conj}[General form of the constraints]
For $a$ positive odd integer and $k\geq0$ integer we~have
\be
\sum_{n=1}^\infty \sigma_a(n) \int_{\Re(s)=\sigma} \Bigg[ - \frac{g(s)g(s-a)}{n^s} (-s)^{k} + \frac{g(s)g(s+a)}{n^{s+a}} (s-1)^{k} \Bigg] ds =0,
\ee
where the integration contour is to the right of the poles.
\end{conj}

When $a$ is an even integer, divisor function constraint relations of the type in Theorem \ref{thm3ishere} still hold, however now the summation kernel is a combination of higher order transcendental functions, rather than a rational function times a decaying exponential. This happens because the poles entering the computation of the inverse Mellin transform are now double poles. We present below an example of such a relation, for $a=0$. In this case $\sigma_0(n)=d(n)$, the number of divisors of $n$.

\begin{problem}
\label{conj2}
For $a=0$ we have
\be
\label{eq51conject}
\sum_{n=1}^\infty \sigma_0(n) n^2 \left[\left(84 \pi^2 n^2+45\right) K_0(2 \pi n )-16 \pi  n \left(\pi ^2 n^2+6\right) K_1(2\pi n )\right]  = 0,
\ee
with $K_\alpha(z)$ the modified Bessel function of the second kind.
\end{problem}

Note that when $n$ is large the summand in Eq. \eqref{eq51conject} asymptotically decreases as $e^{-2\pi n}$ times a polynomial of degree $9$ in~$\sqrt{n}$. As numerical evidence, we present in Table \ref{ourtable} values of the sum in Eq. \eqref{eq51conject}, computed with $n$ running from $1$~up to a large integer $N$, for several values of $N$ up to $N=130$. We denote this partial sum by $S_{\sigma_0}(N)$. It can be been from the table that as $N$ increases the sum has rapid decay.

\begin{table}[htp]
\begin{tabular}{c||c|c|c|c|c}
	$N$ & 10 & 40 & 70 & 100 & 130\\
	\hline
	$S_{\sigma_0(N)}$ & $2.0\cdot 10^{-23}$ & $1.1\cdot 10^{-102}$ & $1.9\cdot 10^{-183}$  & $1.3 \cdot 10^{-264}$  & $5.7\cdot 10^{-346}$
\end{tabular}
\caption{The sum in Eq. \eqref{eq51conject}, considered up to a large but finite value $N$ of the summation upper bound.\label{ourtable}}
\end{table}

Similar relations to Eq. \eqref{eq51conject} can be derived for odd values of $a$ as well, however for $a>5$ odd the functions that appear in the exponential sums become rational, rather than polynomial. We give several examples below.

\begin{problem} 
For $a=7,9,11$  one has
\begin{align}
	\sum_{n=1}^\infty \frac{\sigma_7(n)e^{-2 \pi  n}}{n^5} \lb 16 \pi ^6 n^6-32 \pi ^5 n^5-60 \pi ^4 n^4-90 \pi ^3
   n^3-105 \pi ^2 n^2-90 \pi  n-45\rb &= 0,\nonumber \\
	\sum_{n=1}^\infty \frac{\sigma_9(n) e^{-2 \pi  n} }{n^7}
 \left(32 \pi ^7 n^7-168 \pi ^5 n^5-588 \pi ^4 n^4-1260 \pi ^3 n^3-1890 \pi ^2
   n^2-1890 \pi  n-945\right) &= 0,\nn \\
    \sum_{n=1}^\infty \frac{\sigma_{11}(n) e^{-2 \pi  n} }{n^9}\big( 64 \pi ^8 n^8+160 \pi ^7 n^7-96 \pi ^6 n^6-2016 \pi ^5 n^5-7980 \pi ^4 n^4 \nonumber \\
 \ -19530 \pi ^3 n^3-32445 \pi ^2 n^2-34020 \pi  n-17010  \big) &= 0.
\end{align}
\end{problem}

\begin{problem}
For $a=7,9,11$ one has
\begin{align}
	\sum_{n=1}^\infty \frac{\sigma_7(n)e^{-2\pi n}}{\pi^5 n^5} \left(\pi ^2 n^2+1\right) \left(4 \pi ^2 n^2+6 \pi n+3\right) &= \frac{1}{225}, \nonumber \\
	\sum_{n=1}^\infty \frac{\sigma_9(n)e^{-2\pi n}}{\pi^7 n^7} \lb 8 \pi ^5 n^5+28 \pi ^4 n^4+60 \pi ^3 n^3+90 \pi ^2 n^2+90\pi  n+45 \rb &= \frac{4}{693}, \nonumber \\
	\sum_{n=1}^\infty \frac{\sigma_{11}(n)e^{-2\pi n}}{\pi^9 n^9} \big( 16 \pi ^6 n^6+96 \pi ^5 n^5+324 \pi ^4 n^4+750 \pi ^3 n^3\nn\\
\ +1215 \pi ^2 n^2+1260 \pi  n+630 \big) &= \frac{11056}{1289925}.
\end{align}
\end{problem}

Finally, we remark that numerically it appears to be the case that the family of constraints in Eq.~\eqref{eq41theseareconstraints} uniquely determine the divisor functions $\sigma_{1,3,5}(n)$ for all $n>1$. This provides an alternative method for computing the divisor functions, without employing divisibility techniques coming from the definition of $\sigma_a$.

The procedure is as follows. Suppose we are interested in determining the function $\sigma_1(n)$ for all positive integers $2\leq n\leq M$, for a large integer $M$. Relations \eqref{eq41theseareconstraints} for $a=1$ are a set of linear constraints between all values of the divisor function $\sigma_1(n)$. Each relation involves values $\sigma_1(n)$ for all $n$, however one can truncate in an advantageous manner to $n\leq N$ for an integer $N>M$. Then consider the $M-1$ relations corresponding to $k=0,\dots, M-2$ and solve for the divisor function values $\sigma_1(n)$, $ 2\leq n\leq M$ by inverting the linear system of equations. Doing so will yield approximate values for $\sigma(n)$, with $2\leq n\leq N$, however numerically one can observe that these approximate values converge to the correct ones as $N$ is increased (with $M$ held fixed). In particular, these approximate values tend to integer values.

To be more concrete, relations \eqref{eq41theseareconstraints} for $a=1$ are of the form
\be
\lcb 1 + \sum_{n=2}^N \sigma_1(n) Q_{k}(2\pi n) e^{-2\pi n} + \sum_{n=N+1}^\infty \sigma_1(n) Q_{k}(2\pi n) e^{-2\pi n} = 0 \rcb_{k\geq 0}
\ee
where we have normalized by the first term such that 
\be
Q_k(2\pi n) \coloneqq \frac{P_{k}(2\pi n)}{P_{k}(2\pi) e^{-2\pi}}.
\ee
We can solve this system for $\sigma_1(n)$ for $1<n\leq N$ by matrix inversion, obtaining the solutions as
\be
\label{eq56bold}
\boldsymbol{\sigma}_1 = - {\bf Q}_N^{-1} \lb {\bf 1} + {\bf e}_N \rb, 
\ee
where $\mathbf{Q}^{-1}_N$ is the inverse of matrix
\be
\label{eq5p7matrix}
\mathbf{Q}_N \coloneqq \lcb Q_k(2\pi n) e^{-2\pi n} \rcb_{\substack{0\leq k \leq N-2\\ 1< n \leq N}}
\ee
and $\mathbf{e}_N$ is the vector 
\be
\mathbf{e}_N \coloneqq \lcb \sum_{n=N+1}^\infty \sigma_1(n) Q_k(2\pi n) e^{-2\pi n} \rcb_{0\leq k\leq N-2}. 
\ee
Note that in order for Eq. \eqref{eq56bold} to make sense, matrix ${\bf Q}_N$ must be invertible. We will prove this in Lemma~\ref{lemmadeterm} below.

We can approximate Eq. \eqref{eq56bold} by ignoring the term proportional to ${\bf e}_N$ and writing
\be
\boldsymbol{\sigma}^\mrm{approx}_1 = - {\bf Q}_N^{-1} {\bf 1},
\ee
where $\boldsymbol{\sigma}^\mrm{approx}_1$ is a vector that will approximate the divisor function values in the following precise sense. Consider the first $M-1$ values in $\boldsymbol{\sigma}^\mrm{approx}_1$: as $N$ is increased (with $M$ held fixed) these first $M-1$ values will approach the true values of the divisor function $\sigma_1(n)$, $2\leq n\leq M$.

We formalize this observation in the following conjecture.

\begin{conj}
\label{conj4inversion}
Let integers $M,N$ and matrix ${\bf Q}_N$ as above. Then the first $M-1$ entries in
\be
\label{eq510}
- {\bf Q}_N^{-1} {\bf 1}
\ee
tend to $\sigma_1(n)$, $2 \leq  n\leq M$, as $N\to\infty$.
\end{conj}

As numerical evidence, we present in Tables \ref{tab2} and \ref{tab3} the divisor function values $\sigma^\mrm{approx}_1(n)$ computed according to Eq. \eqref{eq56bold}, for $2\leq n\leq M= 20$ and $2\leq n\leq M = 50$, respectively, with $30$ significant figures. In both cases we chose the values of $M-1$ to be exactly equal to the number of linear equations that we solve. Note that for $M=20$, $\sigma^\mrm{approx}_1(n)$ precisely equals the exact divisor function values, to within 30 significant figures, for $n\leq 6$. For $M=50$, $\sigma^\mrm{approx}_1(n)$ equals the exact divisor function values, to within 30 significant figures, for $n\leq 32$.

\begin{table}[htp]
\begin{center}
\begin{tabular}{ |c|c| } 
 \hline
 $\sigma(2) = 3.00000000000000000000000000000$ & $\sigma(12) = 27.9999999999999996400890699949$ \\
 $\sigma(3) = 4.00000000000000000000000000000$ & $\sigma(13) = 14.0000000000000421206228597205$ \\
 $\sigma(4) = 7.00000000000000000000000000000$ & $\sigma(14) = 23.9999999999956031478005095380$ \\
 $\sigma(5) = 6.00000000000000000000000000000$ & $\sigma(15) = 24.0000000004038437551262919506$ \\
 $\sigma(6) = 12.0000000000000000000000000000$ & $\sigma(16) = 30.9999999679068197920815890622$ \\
 $\sigma(7) = 8.00000000000000000000000000422$ & $\sigma(17) = 18.0000021596112096215771414321$ \\
 $\sigma(8) = 14.9999999999999999999999992457$ & $\sigma(18) = 38.9998804716445031492998761128$ \\
 $\sigma(9) = 13.0000000000000000000001259399$ & $\sigma(19) = 20.0052181778670231340259446064$ \\
 $\sigma(10) = 17.9999999999999999999805051397$ & $\sigma(20) = 41.8318094963373465161153547796$ \\
 $\sigma(11) = 12.0000000000000000027747630780$ & $\sigma(21) = 35.5443381579433006144710767469$ \\
 \hline
\end{tabular}
\end{center}
\caption{The approximate divisor function values $\sigma^\mrm{approx}_1(n)$, computed from Eq.~\eqref{eq56bold}, for $M=N=20$. Note that the error increases as $n$ approaches $M$, however rounding $\sigma^\mrm{approx}_1(n)$  to the closest integer produces the correct divisor function values for all $n\leq~20$. Note that $\sigma_1(21)=32$.}
\label{tab2}
\end{table}

\begin{table}[htp]
\begin{center}
\begin{tabular}{ |c|c| } 
 \hline
 $\sigma(2) = 3.00000000000000000000000000000$ & $\sigma(27) = 40.0000000000000000000000000000$ \\
 $\sigma(3) = 4.00000000000000000000000000000$ & $\sigma(28) = 56.0000000000000000000000000000$ \\
 $\sigma(4) = 7.00000000000000000000000000000$ & $\sigma(29) = 30.0000000000000000000000000000$ \\
 $\sigma(5) = 6.00000000000000000000000000000$ & $\sigma(30) = 72.0000000000000000000000000000$ \\
 $\sigma(6) = 12.0000000000000000000000000000$ & $\sigma(31) = 32.0000000000000000000000000000$ \\
 $\sigma(7) = 8.00000000000000000000000000000$ & $\sigma(32) = 63.0000000000000000000000000000$ \\
 $\sigma(8) = 15.0000000000000000000000000000$ & $\sigma(33) = 48.0000000000000000000000000001$ \\
 $\sigma(9) = 13.0000000000000000000000000000$ & $\sigma(34) = 53.9999999999999999999999999862$ \\
 $\sigma(10) = 18.0000000000000000000000000000$ & $\sigma(35) = 48.0000000000000000000000012260$ \\
 $\sigma(11) = 12.0000000000000000000000000000$ & $\sigma(36) = 90.9999999999999999999998979327$ \\
 $\sigma(12) = 28.0000000000000000000000000000$ & $\sigma(37) = 38.0000000000000000000079576281$ \\
 $\sigma(13) = 14.0000000000000000000000000000$ & $\sigma(38) = 59.9999999999999999994213354568$ \\
 $\sigma(14) = 24.0000000000000000000000000000$ & $\sigma(39) = 56.0000000000000000390676790759$ \\
 $\sigma(15) = 24.0000000000000000000000000000$ & $\sigma(40) = 89.9999999999999975640254651705$ \\
 $\sigma(16) = 31.0000000000000000000000000000$ & $\sigma(41) = 42.0000000000001394294384620250$ \\
 $\sigma(17) = 18.0000000000000000000000000000$ & $\sigma(42) = 95.9999999999927259608230920320$ \\
 $\sigma(18) = 39.0000000000000000000000000000$ & $\sigma(43) = 44.0000000003429798458144049557$ \\     
 $\sigma(19) = 20.0000000000000000000000000000$ & $\sigma(44) = 83.9999999855323327138581857528$ \\
 $\sigma(20) = 42.0000000000000000000000000000$ & $\sigma(45) = 78.0000005391223423057227252954$ \\
 $\sigma(21) = 32.0000000000000000000000000000$ & $\sigma(46) = 71.9999825342148337933538683090$ \\   
 $\sigma(22) = 36.0000000000000000000000000000$ & $\sigma(47) = 48.0004817163156054254463598444$ \\
 $\sigma(23) = 24.0000000000000000000000000000$ & $\sigma(48) = 123.989009643563883194794288057$ \\
 $\sigma(24) = 60.0000000000000000000000000000$ & $\sigma(49) = 57.1989083146438762611930377305$ \\   
 $\sigma(25) = 31.0000000000000000000000000000$ & $\sigma(50) = 90.3291797370224242105622583438$ \\
 $\sigma(26) = 42.0000000000000000000000000000$ & $\sigma(51) = 95.4859364517319426901114775547$ \\  
 \hline
\end{tabular}
\end{center}
\caption{The approximate divisor function values $\sigma^\mrm{approx}_1(n)$, computed from Eq.~\eqref{eq56bold}, for $M=N=50$. Rounding  to the closest integer $\sigma^\mrm{approx}_1(n)$ produces the correct divisor function values for all $n\leq~49$. Note that $\sigma_1(50)=93$ and $\sigma_1(51) = 72$.}
\label{tab3}
\end{table}

\begin{lemma}
\label{lemmadeterm}

The matrix ${\bf Q}_N$ defined in Eq. \eqref{eq5p7matrix} is invertible for all $N\geq 2$.

\end{lemma}

\begin{proof}
Consider the matrix ${\bf q}_N$ as a function of a real variable $x$, that is
\be
\label{eq511matrixq}
\mathbf{q}_N(x) \coloneqq \lcb Q_k(x n) e^{-2\pi n} \rcb_{\substack{0\leq k \leq N-2\\ 1< n \leq N}},
\ee
so that 
\be
\mathbf{Q}_N = \mathbf{q}_N(2\pi)
\ee
for matrix $\mathbf{Q}_N$ defined in Eq. \eqref{eq5p7matrix}. Then $\mathbf{Q}_N$ is invertible if and only if $\mathbf{q}_N(2\pi)$ is invertible. Note that the factors of $e^{-2\pi n}$ in Eq. \eqref{eq511matrixq} are the same along the columns of the matrix, so they factor out of the determinant of matrix $\mathbf{q}_N(2\pi)$. Then $\mathbf{q}_N(2\pi)$ is invertible if and only if the matrix $\hat{\mathbf{q}}_N(2\pi)$ is invertible,~with
\be
\hat{\mathbf{q}}_N(x) \coloneqq \lcb Q_k(x n)\rcb_{\substack{0\leq k \leq N-2\\ 1< n \leq N}}.
\ee
The functions $Q_k(xn)$ are polynomials of degree $2k+4$ with rational coefficients, so the determinant of  $\hat{\mathbf{q}}_N(x)$ will be a polynomial in $x$ with rational coefficients. Because $\pi$ is transcendental and the coefficients are rational, this polynomial will vanish at $x=2\pi$ if and only if it is identically zero. However, from Eq.~\eqref{eq42threecases} the monomial of lowest exponent in $Q_k(x n)$ is $ (-nx)^{k+1}$. These monomials will give the contribution of lowest exponent in $x$ to $\det \lb \hat {\bf q}_N (x) \rb$, so this contribution will be proportional to the determinant of the Vandermonde matrix 
\be
V(N) \coloneqq \lcb n^{k+1} \rcb_{\substack{0\leq k \leq N-2\\ 1< n \leq N}}.
\ee 
This matrix has determinant given by
\begin{align}
\det(V(N)) = N!\prod_{k=1}^{N-2}k!,
\end{align}
and so it is never vanishing. Thus, $\det \lb \hat {\bf q}_N (x) \rb$ is not identically zero as a polynomial in $x$ and the result follows.
\end{proof}

\begin{table}[htp]
\begin{center}
\begin{tabular}{ |c|c| } 
 \hline
 $\sigma_3(2) = 9.00000000000000000000000000000$ & $\sigma_3(27) = 20440.0000000000000000000000000$ \\
 $\sigma_3(3) = 28.0000000000000000000000000000$ & $\sigma_3(28) = 25112.0000000000000000000000000$ \\
 $\sigma_3(4) = 73.0000000000000000000000000000$ & $\sigma_3(29) = 24390.0000000000000000000000000$ \\
 $\sigma_3(5) = 126.000000000000000000000000000$ & $\sigma_3(30) = 31752.0000000000000000000000000$ \\
 $\sigma_3(6) = 252.000000000000000000000000000$ & $\sigma_3(31) = 29792.0000000000000000000000000$ \\
 $\sigma_3(7) = 344.000000000000000000000000000$ & $\sigma_3(32) = 37449.0000000000000000000000000$ \\
 $\sigma_3(8) = 585.000000000000000000000000000$ & $\sigma_3(33) = 37296.0000000000000000000000002$ \\
 $\sigma_3(9) = 757.000000000000000000000000000$ & $\sigma_3(34) = 44225.9999999999999999999999843$ \\
 $\sigma_3(10) = 1134.00000000000000000000000000$ & $\sigma_3(35) = 43344.0000000000000000000014357$ \\
 $\sigma_3(11) = 1332.00000000000000000000000000$ & $\sigma_3(36) = 55260.9999999999999999998770669$ \\
 $\sigma_3(12) = 2044.00000000000000000000000000$ & $\sigma_3(37) = 50654.0000000000000000098501508$ \\
 $\sigma_3(13) = 2198.00000000000000000000000000$ & $\sigma_3(38) = 61739.9999999999999992644180773$ \\
 $\sigma_3(14) = 3096.00000000000000000000000000$ & $\sigma_3(39) = 61544.0000000000000509622777812$ \\
 $\sigma_3(15) = 3528.00000000000000000000000000$ & $\sigma_3(40) = 73709.9999999999967414434146483$ \\
 $\sigma_3(16) = 4681.00000000000000000000000000$ & $\sigma_3(41) = 68922.0000000001911315728467208$ \\
 $\sigma_3(17) = 4914.00000000000000000000000000$ & $\sigma_3(42) = 86687.9999999897884814899862414$ \\
 $\sigma_3(18) = 6813.00000000000000000000000000$ & $\sigma_3(43) = 79508.0000004927595760801165753$ \\     
 $\sigma_3(19) = 6860.00000000000000000000000000$ & $\sigma_3(44) = 97235.9999787415909482729911794$ \\
 $\sigma_3(20) = 9198.00000000000000000000000000$ & $\sigma_3(45) = 95382.0008096383425947318643654$ \\
 $\sigma_3(21) = 9632.00000000000000000000000000$ & $\sigma_3(46) = 109511.973211472559620985052821$ \\   
 $\sigma_3(22) = 11988.0000000000000000000000000$ & $\sigma_3(47) = 103824.753967741976422710694216$ \\
 $\sigma_3(23) = 12168.0000000000000000000000000$ & $\sigma_3(48) = 131050.463927831941005026093101$ \\
 $\sigma_3(24) = 16380.0000000000000000000000000$ & $\sigma_3(49) = 118316.076980132304317772971012$ \\   
 $\sigma_3(25) = 15751.0000000000000000000000000$ & $\sigma_3(50) = 137354.040515591541150389821587$ \\
 $\sigma_3(26) = 19782.0000000000000000000000000$ & $\sigma_3(51) = 176690.854995111802482932995631$ \\  
 \hline
\end{tabular}
\end{center}
\caption{The approximate divisor function values $\sigma^\mrm{approx}_3(n)$, computed from the analogue of Eq.~\eqref{eq56bold} for $a=3$, for $M=N=50$. Rounding  to the closest integer $\sigma^\mrm{approx}_3(n)$ produces the correct divisor function values for all $n\leq~46$.}
\label{tab4}
\end{table}

\begin{table}[htp]
\begin{center}
\begin{tabular}{ |c|c| } 
 \hline
 $\sigma_5(2) = 33.0000000000000000000000000000$ & $\sigma_5(27) = 14408200.0000000000000000000000$ \\
 $\sigma_5(3) = 244.000000000000000000000000000$ & $\sigma_5(28) = 17766056.0000000000000000000000$ \\
 $\sigma_5(4) = 1057.00000000000000000000000000$ & $\sigma_5(29) = 20511150.0000000000000000000000$ \\
 $\sigma_5(5) = 3126.00000000000000000000000000$ & $\sigma_5(30) = 25170552.0000000000000000000000$ \\
 $\sigma_5(6) = 8052.00000000000000000000000000$ & $\sigma_5(31) = 28629152.0000000000000000000000$ \\
 $\sigma_5(7) = 16808.0000000000000000000000000$ & $\sigma_5(32) = 34636833.0000000000000000000000$ \\
 $\sigma_5(8) = 33825.0000000000000000000000000$ & $\sigma_5(33) = 39296688.0000000000000000000003$ \\
 $\sigma_5(9) = 59293.0000000000000000000000000$ & $\sigma_5(34) = 46855313.9999999999999999999744$ \\
 $\sigma_5(10) = 103158.000000000000000000000000$ & $\sigma_5(35) = 52541808.0000000000000000024043$ \\
 $\sigma_5(11) = 161052.000000000000000000000000$ & $\sigma_5(36) = 62672700.9999999999999997881973$ \\
 $\sigma_5(12) = 257908.000000000000000000000000$ & $\sigma_5(37) = 69343958.0000000000000174458929$ \\
 $\sigma_5(13) = 371294.000000000000000000000000$ & $\sigma_5(38) = 81711299.9999999999986617370992$ \\
 $\sigma_5(14) = 554664.000000000000000000000000$ & $\sigma_5(39) = 90595736.0000000000951717342883$ \\
 $\sigma_5(15) = 762744.000000000000000000000000$ & $\sigma_5(40) = 105736949.999999993757817615787$ \\
 $\sigma_5(16) = 1082401.00000000000000000000000$ & $\sigma_5(41) = 115856202.000000375328863008164$ \\
 $\sigma_5(17) = 1419858.00000000000000000000000$ & $\sigma_5(42) = 135338015.999979456901587961233$ \\
 $\sigma_5(18) = 1956669.00000000000000000000000$ & $\sigma_5(43) = 147008444.001014948920706266757$ \\     
 $\sigma_5(19) = 2476100.00000000000000000000000$ & $\sigma_5(44) = 170231963.955195845427215352208$ \\
 $\sigma_5(20) = 3304182.00000000000000000000000$ & $\sigma_5(45) = 185349919.745044668152688836199$ \\
 $\sigma_5(21) = 4101152.00000000000000000000000$ & $\sigma_5(46) = 212399292.987763796989448216956$ \\   
 $\sigma_5(22) = 5314716.00000000000000000000000$ & $\sigma_5(47) = 229346704.548832254245463654739$ \\
 $\sigma_5(23) = 6436344.00000000000000000000000$ & $\sigma_5(48) = 264065564.199393368665742826277$ \\
 $\sigma_5(24) = 8253300.00000000000000000000000$ & $\sigma_5(49) = 283249010.702784989739509435916$ \\   
 $\sigma_5(25) = 9768751.00000000000000000000000$ & $\sigma_5(50) = 311853108.670425831350122598682$ \\
 $\sigma_5(26) = 12252702.0000000000000000000000$ & $\sigma_5(51) = 441339084.443169397775938475135$ \\  
 \hline
\end{tabular}
\end{center}
\caption{The approximate divisor function values $\sigma^\mrm{approx}_5(n)$, computed from the analogue of Eq.~\eqref{eq56bold} for $a=5$, for $M=N=50$. Rounding to the closest integer $\sigma^\mrm{approx}_5(n)$ produces the correct divisor function values for all $n\leq~44$.}
\label{tab5}
\end{table}

Numerical experimentation suggests that the method in Conjecture~\ref{conj4inversion} works to determine the divisor functions $\sigma_a(n)$ for other values of $a$ as well, such as $a=3,5$, when one replaces the $Q_k(x)$ polynomials with the appropriate expressions from Eq. \eqref{eq41theseareconstraints} (see Tables \ref{tab4} and  \ref{tab5}). It would be interesting to explore these issues systematically.

Finally, we would like to discuss how the ideas presented above could relate to already existing work. The connection between $\zeta^2(s), \sigma_0(n)=d(n)$ and $K_{0}(s)$ is natural, since we have
\be
\zeta^2(s) = \sum_{n=1}^\infty \frac{d(n)}{n^s}, \quad \Re(s) >1,
\ee 
and a well-known Mellin transform relates the $K_a(x)$ Bessel function to a product of two Gamma functions, 
\begin{align}
	\int_0^\infty K_a(x) x^{s-1}dx = 2^{s-2}\Gamma\bigg(\frac{s-a}{2}\bigg)\Gamma\bigg(\frac{s+a}{2}\bigg).
\end{align}
Koshliakov considered infinite sums involving $\sigma_0(n)$ twisted by the Bessel function $K_0$. These are given in terms of integrals involving the square of $\Xi(t)=\xi(\frac{1}{2}+it)$. For instance he was able to show \cite{Koshliakov1, Koshliakov2, Koshliakov3} that for $\theta>0$ one has that
\begin{align*}
-\frac{32}{\pi} \int_0^\infty \Xi^2\bigg(\frac{t}{2}\bigg) \cos\bigg(\frac{1}{2}t \log \theta \bigg) \frac{dt}{(1+t^2)^2} = \sqrt{\theta} \bigg(\frac{\gamma-\log(4\pi \theta)}{\theta} - 4 \sum_{n=1}^\infty \sigma_0(n) K_0 (2\pi n \theta)\bigg),
\end{align*}
where $\gamma$ is the Euler constant. Ramanujan \cite{Ramanujan1, Ramanujan2} also showed that for $\theta>0$
\begin{align*}
-\frac{1}{\pi^{3/2}} \int_0^\infty \bigg| \Gamma\bigg( \frac{-1+it}{4}\bigg) \bigg|^2 \Xi^2\bigg(\frac{t}{2}\bigg)\cos\bigg(\frac{1}{2}t \log \theta \bigg) \frac{dt}{1+t^2} = \sqrt{\theta}\bigg(\frac{\gamma-\log(2\pi \theta)}{2\theta} + \sum_{n=1}^\infty \lambda(n\theta)\bigg),
\end{align*}
where the function $\lambda$ is given by
\begin{align}
\lambda(x) := \frac{\Gamma'}{\Gamma}(x) + \frac{1}{2x} - \log x.
\end{align}
The arithmetic link to $\sigma_0(n)$ comes from, see \cite[Eq. (1.12)]{Dixit1},
\begin{align}
   - \frac{1}{2}\sum_{n = 1}^\infty \lambda (n\theta ) &= \int_0^\infty x{e^{ - 2\pi \theta x}}\sum_{n = 1}^\infty  \frac{\sigma _0(n)}{x^2 + n^2}dx \nonumber \\ 
   &= \sum_{n = 1}^\infty \sigma _0(n)\left[ \frac{1}{2}\sin (2\pi n\theta )\left( \pi  - 2\operatorname{Si} (2\pi n\theta ) \right) - \cos (2\pi n\theta )\operatorname{Ci} (2\pi n\theta ) \right],
\end{align}
where the cosine and sine integral are defined by
\begin{align}
  \operatorname{Ci} (x) =  - \int_x^\infty   \frac{\cos t}{t}dt\quad \textnormal{and} \quad \operatorname{Si} (x) = \int_0^x {} \frac{\sin t}{t}dt.
\end{align}
In the 2000's and 2010's Berndt and Dixit \cite{BerndtDixit, Dixit1, Dixit2} started a profound series of generalizations to Koshliakov's and Ramanujan's identities. Some of the generalizations culminated in \cite{DixitRoblesRoyZaharescu1, DixitRoblesRoyZaharescu2}. Further information on how divisor functions relate to Bessel functions can be found in e.g. \cite[Chapters 3 and 4]{Ivic} as well as in \cite[Chapter 12]{Titchmarsh}. It would be very interesting to study the potential connections between Problem \ref{conj2} and the Koshliakov and Ramanujan identities. 

Relations for the divisor functions $\sigma_a(n)$ reminiscent to the ones in the present paper were also considered in \cite{BDRZpathways}, for $-1/2 < \Re(a) < 1/2$. It would be interesting to investigate if these relations can be recovered from our general framework.

Lastly, we end by opening the possibility of considering other arithmetic functions such as the M\"{o}bius function $\mu$, the Euler totient function $\varphi$, or the von Mangoldt function $\Lambda$, to name a few. In these situations, the generating functions will need to be adapted. For instance, in the case of the M\"{o}bius function one employs
\begin{align}
	\frac{1}{\zeta(s)} = \sum_{n=1}^\infty \frac{\mu(n)}{n^s}, \quad \Re(s) > 1.
\end{align}
In principle, the techniques of the present paper can also be applied here, with certain natural modifications. The case of the inverse of $\zeta(s)$ requires special care when performing contour integration. Originally, the fact that the horizontal lines of a rectangular contour tend to zero could only be proved under the assumption of the Riemann Hypothesis, \cite[$\mathsection$14 and Eq. (14.16.2)]{Titchmarsh}. However, due to the work of Ramachandra and Sankaranarayanan \cite{RamachandraSankaranarayanan} the necessary tools are now unconditional, see also Bartz \cite{Bartz}, Inoue \cite{Inoue} and K\"{u}hn and Robles \cite{KuhnRobles} for related applications.

\textbf{Acknowledgments.} B. S. would like to acknowledge the Northwestern University Amplitudes and Insights group, Department of Physics and Astronomy, and Weinberg College for support. The work of B. S. was supported in part by the Department of Energy under Award Number DE-SC0021485.

\end{document}